\documentclass{amsart}
\usepackage{amsmath}%
\usepackage{amsfonts}%
\usepackage{amssymb}%
\usepackage{graphicx}
\usepackage{mathtools}

\renewcommand{\H}{\mathcal{H}}
\newcommand{\G}{\mathcal{G}}
\newcommand{\C}{\mathcal{C}}

\newcommand{\B}{\mathcal{B}}

\renewcommand{\O}{\mathcal{O}}

\newcommand{\Mod}{\mbox{Mod}}
\newcommand{\R}{\mathbb{R}}

\newcommand{\h}{\mathbb{H}^2}
\renewcommand{\S}{\mathcal{S}}

\renewcommand{\Im}{\mathbf{Im \,}}
\newtheorem{theorem}{Theorem}[section]
\theoremstyle{plain}
\newtheorem{lem}[theorem]{Lemma}
\newtheorem{claim}[theorem]{Claim}
\newtheorem{cor}[theorem]{Corollary}
\newtheorem{rem}[theorem]{Remark}

\numberwithin{equation}{subsection}
\theoremstyle{definition}
\newtheorem{defi}[theorem]{Definition}

\title[A Birman-Series type result]{A Birman-Series type result for geodesics with infinitely many self-intersections}
\author{Jenya Sapir}

\begin{document}

\begin{abstract}
Given a hyperbolic surface $\S$, a classic result of Birman and Series states that for each $K$, all complete geodesics with at most $K$ self-intersections can only pass through a certain nowhere dense, Hausdorff dimension 1 subset of $\S$. We define a self-intersection function for each complete geodesic, which bounds the number of self-intersections in finite length subarcs. We then extend the Birman-Series result to sets of complete geodesics with certain bounds on their self-intersection functions. In fact, we get the same conclusion as the Birman-Series result for sets of complete geodesics whose self-intersection functions are in $o(l^2)$, where $l$ measures arclength.
\end{abstract}

\maketitle

\section{Introduction}
Let $\S$ be a genus $g$ surface with $n$ boundary components, and let $X$ be a hyperbolic metric on $\S$ in which each boundary component is geodesic. Consider the set $\G$ of complete geodesics on $\S$. In particular, geodesics in $\G$ never hit the boundary of $\S$. 

Given any subset $\H \subset \G$, we define the \textbf{image} of $\H$ in $\S$, denoted $\Im \H$, to be the set of points in $\S$ that lie on some curve in $\H$. 
Birman and Series showed that, for each $K$, if $\H$ is the set of all complete geodesics with at most $K$ self-intersections, then $\Im \H$ is nowhere dense and has Hausdorff dimension 1 \cite{BS85}. 

In this paper, we find much weaker conditions on subsets $\H \subset \G$ that give the same conclusion. We get these conditions by studying the self-intersection function of each $\gamma \in \G$, which is defined as follows. Take a complete geodesic $\gamma : \R \rightarrow \S$ parameterized by arclength. Let 
$
 \gamma_l = \gamma|_{[-\frac l2, \frac l2]}
$
be the length $l$ subarc of $\gamma$ centered at $\gamma(0)$. Then $f(l) = i(\gamma_l, \gamma_l)$ is the \textbf{self-intersection function} of $\gamma$. 

This function depends on the parameterization of $\gamma$, so we choose a parameterization for each $\gamma \in \G$ so that its self-intersection function is as small as possible. That is, if $\gamma : \R \rightarrow \S$ and $\gamma' : \R \rightarrow \S$ are two parameterizations by arclength of the same complete geodesic, then $\gamma$ has a smaller self-intersection function than $\gamma'$ if $i(\gamma_l, \gamma_l) \lesssim i(\gamma'_l, \gamma'_l)$. Note that we write $A(l) \lesssim B(l)$ if $\limsup_{l \rightarrow \infty} \frac{A(l)}{B(l)} \leq 1$. There need not be a parameterization of $\gamma$ with the least self-intersection function, so we make an arbitrary choice of parameterization for each $\gamma \in \G$.

Suppose $f : \R \rightarrow \R$ is any function. Let 
\[
 \G(f) = \{ \gamma \in \G \ | \ i(\gamma_l, \gamma_l) \lesssim f(l)\}
\]

\begin{theorem}
\label{thm:HausDim}
 For any $k > 0$, suppose $f(l)$ is a function with $f(l) \leq (k l)^2$ for all $l$ large enough. Then the Hausdorff dimension of $\Im \G(f)$ is at most $\mu(k)$, where $\lim_{k \rightarrow 0} \mu(k) = 1$.
In particular, if $f(l) = o(l^2)$, then $\Im \G(f)$ has Hausdorff dimension 1. 
\end{theorem}
On the other hand, \cite{LS15} implies that $\Im \G(f)$ is dense whenever $f(l)$ is superlinear in $l$ (see Section \ref{sec:TheoremsForClosed}). Nevertheless, we get a nowhere density result once we get more control on the self-intersection function of our geodesics. So if $f : \R \rightarrow \R$ is any function, then let 
\[
 \G(f,L) = \{ \gamma \in \G \ | \ i(\alpha,\alpha) \leq f(l), \forall \alpha \subset \gamma \text{ s.t. } l(\alpha) \geq L\}
\]
be the set of $\gamma \in \G$ such that all length $l$ subarcs have at most $f(l)$ self-intersections, whenever $l \geq L$. 
\begin{theorem}
\label{thm:NWD}
 There is a $k_0 > 0$ so that if $f(l) \leq (k_0 l)^2$ for all $l$, then $\Im \G(f, L)$ is nowhere dense for all $L \geq 0$.
\end{theorem}

\begin{rem}
 Both the original result of Birman and Series, as well as Theorems \ref{thm:HausDim} and \ref{thm:NWD}, still hold when $\S$ has a negatively curved metric with curvature bounded away from zero and infinity. However, the function $\mu$ and constant $k_0$ will depend on the metric. This is because the results below that use hyperbolic geometry can be proven in the general negative curvature case, but with different constants.
\end{rem}



\subsection{Previous results for complete geodesics}
Complete geodesics on $\S$ satisfy the following dichotomy. On the one hand, when $X$ is a complete, finite volume metric, then $\Im \G = \S$. Even when $X$ has geodesic boundary, $\Im \G$ can have Hausdorff dimension greater than 1, and points of Lebesgue density. In the case of a pair of pants, this is proven in \cite{HJJL12}.

Moreover, when $X$ has finite volume, any ``typical'' geodesic in $\G$ has dense image in $\S$ in the following sense. Let $T_1\S$ be the unit tangent bundle of $\S$. Then we can choose a vector $v\in T_1\S$ at random with respect to Lebesgue measure, and consider the complete geodesic $\gamma$ with tangent vector $\gamma'(0) = v$. By mixing of the geodesic flow, $\Im \gamma$ will be dense with probability 1. Note that in this case, mixing of the geodesic flow also implies that the self-intersection function of $\gamma$ will grow asymptotically like $\kappa l^2$, with probability 1. 

On the other hand, the classical result of Birman and Series that we reference above shows that when $f(l) = K$ is constant, meaning that $\G(f)$ consists of complete geodesics with at most $K$ self-intersections, then $\Im \G(f)$ has Hausdorff dimension 1 and is nowhere dense \cite{BS85}. 

So we can think of self-intersection functions on a sliding scale. On one end of the scale, we have functions with $f(l) = O(1)$, and on the other side, we have functions with $f(l) = O(l^2)$. Theorems \ref{thm:HausDim} and \ref{thm:NWD} allow us to interpolate between these two extremes. They indicate that the transition from Hausdorff dimension 1 to Hausdorff dimension 2, and from being nowhere dense to being dense, occur at the far end of the scale, among functions with $f(l) = O(l^2)$.

\subsection{Contrast with results for closed geodesics}
\label{sec:TheoremsForClosed}

There is an analogous story for closed geodesics. Let $\G^c$ be the set of closed geodesics on $\S$. When $X$ has finite volume (and no boundary), $\Im \G^c$ is dense in $\S$ by the closing lemma and mixing of the geodesic flow. On the other hand, the set of simple closed geodesics has nowhere dense image by \cite{BS85}. 

Recently, Lenzhen and Souto have considered the sets of closed geodesics between these two extremes \cite{LS15}. In particular, for any function $f: \R \rightarrow \R$, they consider the set
\[
 \{ \gamma \in \G^c \ | \ i(\gamma, \gamma) \leq f(l(\gamma))\}
\]
They show that the image of this set is dense whenever $\lim_{l \rightarrow \infty} f(l)/l = \infty$, and that its lift to $T_1\S$ has Hausdorff dimension strictly smaller than 3 if $f(l) = o(l)$.

The first part of their result can be combined with our theorems to get the following corollary.
\begin{cor}[Consequence of \cite{LS15} and Theorem \ref{thm:HausDim}]
If $\lim_{l \rightarrow \infty} f(l)/l = \infty$, then $\Im \G(f)$ is dense. In particular, there is some $k_0 > 0$ so that for any $k < k_0$, $\Im \G(k^2 l^2)$ is dense, but does not have full Hausdorff dimension.
\end{cor}
\begin{proof}
  Suppose  $\lim_{l \rightarrow \infty} f(l)/l = \infty$. This corollary follows from the fact that if $\gamma$ is a closed geodesic with $l(\gamma) = L$ and $i(\gamma, \gamma) \leq f(L)$, then $\gamma \in \G(f)$. 
  
  To see this, note that we can view any closed geodesic $\gamma$ as a complete geodesic $\gamma: \R \rightarrow \S$ parameterized by arclength. If $l(\gamma) = L$, then this parameterization has period $L$. We can define $\gamma_l$ as above to be the length $l$ subarc of $\gamma$ centered at $\gamma(0)$, where $\gamma_l$ is defined for any $l \in \R$. Then
  \[
   \frac{i(\gamma_l, \gamma_l)}{l} \lesssim \frac{f(L)}{L}
  \]
  Because $\lim_{l \rightarrow \infty} f(l)/l = \infty$, it is trivially true that $\lim_{l \rightarrow \infty} \frac{f(L)}{L} \cdot \frac{l}{f(l)} = 0$, since $L$ is a constant. Therefore $i(\gamma_l, \gamma_l) \lesssim f(l)$. In other words, $\gamma \in \G(f)$. 
  
  Thus, $\Im \G(f)$ contains a dense set by \cite{LS15}, and so it is dense.
\end{proof}

On the other hand, a closed geodesic $\gamma$ with $l(\gamma) = L$ and $i(\gamma, \gamma) \leq f(L)$ does not necessarily belong to $\G(f,L')$, if $L' < L$. In fact, to determine whether $\gamma \in \G(f,L')$, one would have to examine how the self-intersections of $\gamma$ are distributed along its length. So \cite{LS15} does not contradict Theorem \ref{thm:NWD}.

It is interesting to note that the transition away from full Hausdorff dimension occurs when $f(l) = (k_0 l)^2$ for complete geodesics, while it occurs around $f(l) = O(l)$ for closures of sets of closed geodesics.

\subsection{Reduction to closed surfaces}
It is enough to prove Theorems \ref{thm:HausDim} and \ref{thm:NWD} for closed surfaces $\S$ (without boundary). In fact, if $X$ has geodesic boundary, then we can double $\S$ across this boundary to get a $\S'$ with finite volume metric $X'$. There is a natural inclusion $\S \hookrightarrow \S'$ along which $X'$ pulls back to $X$. Any $\gamma \subset \S$ gets sent to a geodesic on $\S'$ with the same self-intersection function. So if $\Im \G(f)$ has Hausdorff dimension $h$, or is nowhere dense, on $\S'$, then the same is true on $\S$.

\subsection{Structure of the paper}

  In Section \ref{sec:ArcApprox}, we show how to approximate any complete geodesic $\gamma$ by a sequence of closed geodesics $\gamma_n$ that also approximate the self-intersection function of $\gamma$. In particular, given any geodesic arc $\alpha$ of length $L$, which can be thought of as a subarc of $\gamma$, we show how to find a nearby closed geodesic whose length and self-intersection number are not much larger than those of $\alpha$ (Lemma \ref{lem:ApproxArcs}).
 
  In Section \ref{sec:Covers}, we apply Lemma \ref{lem:ApproxArcs} to construct covers for $\G(f)$ and $\G(f,L)$ that consist of regular neighborhoods of closed geodesics. In particular, for each function $f$, we get a sequence of finite covers $\{\C_n\}$. In Lemma \ref{lem:FiniteCover}, we show that $\C_n$ covers $\G(f,L)$ for all $n$ large enough (depending on $L$.) Moreover, we show that any infinite union of these covers is, in fact, a cover for $\G(f)$ (Lemma \ref{lem:NestedCover}).
 
  In Section \ref{sec:CountingCoverSize}, we approximate the number of open sets in the cover $\C_n$, for each $n$. The set $\C_n$ is a collection of regular neighborhoods of closed geodesics that lie in a certain set.
  So to approximate the size of $\C_n$, we need to approximate the size of this set of closed geodesics. We do this in Lemma \ref{lem:ClosedGeodesicGrowth}.
  
  In Sections \ref{sec:NWD} and \ref{sec:HausDimProof}, we prove Theorems \ref{thm:NWD} and \ref{thm:HausDim}, respectively, given the above lemmas. We do this by getting upper bounds the Lebesgue and Hausdorff measures of each cover $\C_n$, where the measure of a cover is defined to be the measure of the union of elements of that cover.

\subsection{Notation}
There are several points in this paper where we only need coarse estimates. We use the following notation. If two functions $A(x)$ and $B(x)$ satisfy $ A(x) \leq c B(x)$ where $c$ is a constant depending only on some quantity $D$, then we write
\[
A(x) \preccurlyeq B(x)
\]
and say that the constants depend only on $D$. We will also say that $A(x)$ is coarsely bounded by $B(x)$.

Furthermore, given two curves $\alpha$ and $\beta$, $i(\alpha, \beta)$ denotes the least number of self-intersections between all curves freely homotopic to $\alpha$ and $\beta$. On the other hand, $\#\alpha \cap \beta$ denotes the number of transverse intersections between the curves $\alpha$ and $\beta$ themselves. In particular, we use the notation $\#\alpha \cap \beta$ when $\alpha$ and $\beta$ are arcs rather than closed curves.

\section{Lemmas about arcs}
\label{sec:ArcApprox}

It is well-known that one can approximate any complete geodesic with a sequence of closed geodesics. For example, closed geodesics are dense in the space of geodesic currents, which also contain the set of complete geodesics \cite{Bon}. So given any complete geodesic $\gamma_\infty$, we can find a sequence $\{\gamma_i\} \subset \G^c$ so that $\lim \gamma_i = \gamma_\infty$. Note that this limit holds in the measure-theoretic sense of geodesic currents, but it must also hold as a Hausdorff limit of the geodesics themselves. See \cite{Bon} for more details.

Suppose $\gamma_\infty \in \G(f)$. Then not only do we want to approximate $\gamma_\infty$ by a sequence $\{\gamma_i\}$ of closed geodesics, we want the self-intersection numbers of the closed geodesics to eventually be coarsely bounded by the self-intersection function of $\gamma_\infty$. That is, if $l(\gamma_i) = l_i$, we want $i(\gamma_i, \gamma_i) \preccurlyeq f(l_i)$, where the constant is independent of $i$.

We make this precise in terms of subarcs of complete geodesics. In particular, given a point $x \in \Im \gamma_\infty$, we can take a nested sequence of subarcs $\alpha_l \subset \gamma_\infty$ centered at $x$ so that $\alpha_l$ has length $l$. Then the following lemma says we can find a closed geodesic close to $\alpha_l$ that does not have too much more length or to many more self-intersections.

\begin{defi}
 We say a closed geodesic $\gamma$ \textbf{$r$-fellow travels} a geodesic arc $\alpha$ if there are some lifts $\tilde \gamma$ and $\tilde \alpha$ of $\gamma$ and $\alpha$, respectively, to the universal cover $\tilde \S$ of $\S$ so that $\tilde \alpha$ lies in a $r$ -neighborhood of $\tilde \gamma$.
\end{defi}

Let 
\[
 \G^c_L(K) = \{ \gamma \in \G^c \ | \ l(\gamma) \leq L, i(\gamma,\gamma) \leq K\}
\]
be the set of closed geodesics with length at most $L$ and with at most $K$ self-intersections.

\begin{lem}
\label{lem:ApproxArcs}
There is a constant $d$ depending only on the metric $X$ so that the following holds. Let $\alpha$ be a geodesic arc of length $L \geq d$ with $\# \alpha \cap \alpha = K$. Then there exists a closed geodesic 
 \[
 \gamma \in \G^c_{2L}( K + dL)
 \] 
that $1$-fellow-travels $\alpha$.
\end{lem}

This Lemma is a direct consequence of Claims \ref{cla:ArcIntersection} and \ref{cla:ClosingAlpha} below. The proof of Lemma \ref{lem:ApproxArcs} given these claims is at the end of this section.

\begin{claim}
\label{cla:ClosingAlpha}
For any geodesic arc $\alpha$ with $l(\alpha) \geq 3$ there is a $\gamma \in \G^c$ so that $\gamma$ 1-fellow travels $\alpha$ and $l(\gamma) \leq l(\alpha) + R$, where $R$ is a constant depending only on $X$.
 \end{claim}
We would like to thank Chris Leininger for suggesting the idea for this claim and its proof.
\begin{proof}
 Suppose $\beta$ is a geodesic arc so that $\alpha$ and $\beta$ can be concatenated into a closed curve $\gamma'$. Then $\gamma'$ is a piecewise geodesic closed curve with corners at the endpoints of $\alpha$. Suppose the angle deficit at each corner is at most $\epsilon$ (Figure \ref{fig:AngleEpsilon}). 
 \begin{figure}[h!]
  \centering 
  \includegraphics{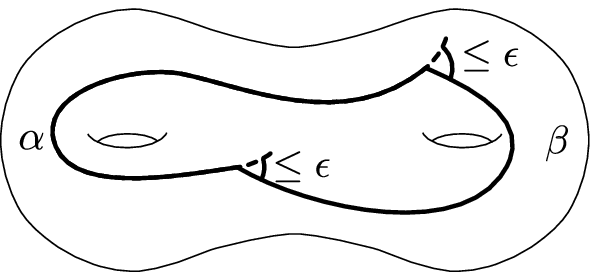}	
  \caption{}
  \label{fig:AngleEpsilon}
 \end{figure}
 Let $\gamma$ be the geodesic representative of $\gamma'$. Then $\gamma$ must $D(\epsilon)$- fellow travel $\gamma'$, for a function $D(\epsilon)$ depending only on the maximal angle deficit $\epsilon$ with 
  \[
  \lim_{\epsilon \rightarrow 0} D(\epsilon) = 0
  \]
 
 To see this, we will round the corners of $\gamma'$ to get a nearby, piecewise $C^2$ curve $\gamma_c$ that $\epsilon \sinh(1)$-fellow travels $\gamma$. The curve $\gamma_c$ will have geodesic curvature bounded by a function $g(\epsilon)$ at each point, where $\lim_{\epsilon \rightarrow 0} g(\epsilon) = 0$. So by \cite{Leininger06}, we can conclude that $\gamma$ must $f(\epsilon)$- fellow travel $\gamma_c$, where $f(\epsilon)$ is a continuous function in $\epsilon$, with $f(0) = 0$.
 
 First, if $l(\beta) \leq 2$, we need to modify $\gamma'$ slightly: Replace $\gamma'$ by the curve $\gamma''$ that is freely homotopic to it relative one of the endpoints of $\alpha$. There are lifts $\tilde \gamma'$ and $\tilde \gamma''$ to the universal cover $\tilde \S$ of $\S$ that form a geodesic triangle $\triangle abc$ where $c$ is the vertex opposite $\tilde \gamma''$ and has angle at least $\pi - \epsilon$ (Figure \ref{fig:Triangle}). Applying the hyperbolic law of sines, we see that the distance from $\tilde \gamma'$ to $\tilde \gamma''$ is at most $\epsilon \sinh(2)$.
 
 
 \begin{figure}[h!]
  \centering 
  \includegraphics{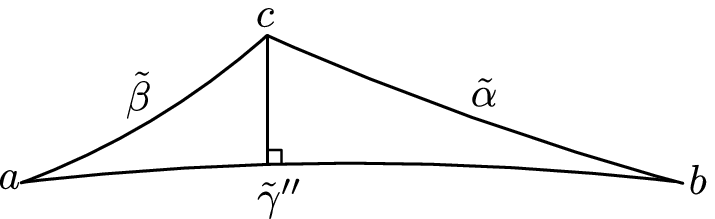}
  \caption{}
  \label{fig:Triangle}
 \end{figure}

 The rest of the proof is essentially the same, whether we deal with $\gamma'$ or $\gamma''$. So assume for what follows that $l(\beta) > 2$, so we deal with $\gamma'$. Take a bi-infinite lift $\tilde \gamma'$ of $\gamma'$ to $\tilde \S$. 
 
 The curve $\tilde \gamma'$ is a piecewise geodesic with angle deficit at most $\epsilon$ at its corners. We will now find a nearby piecewise $C^2$ curve $\tilde \gamma_c$ whose curvature is bounded above by $g(\epsilon)$ at each point, where $\lim_{x \rightarrow 0} g(x) = 0$. 
 
 
 For this, we use the upper half plane model for $\h$. Since $X$ is a hyperbolic metric, we can view $\tilde \S$ as a subset of $\h$. Applying a hyperbolic isometry, we can assume that $\tilde \gamma'$ has a geodesic segment from some point $iy$ to the point $i$, and that it then turns by an angle $\theta < \epsilon$ and has a length 1 geodesic segment from $i$ to some point $b$. Note that this is possible since we assume that $l(\alpha), l(\beta) \geq 2$. Then the segment from $i$ to $b$ lies on a circle with center $a \in \R$ and radius $r$, where
 \[
  a = \frac{1}{\tan \theta}, \quad r = \frac{1}{\sin \theta}
 \]
 In particular, $b = a + re^{i(\pi - \theta - \Delta \theta)}$ where $\Delta \theta$ goes to zero as $\theta$ goes to zero (left side of Figure \ref{fig:DeltaC}).
 
%
%
 
 \begin{figure}[h!]
  \centering 
  \includegraphics{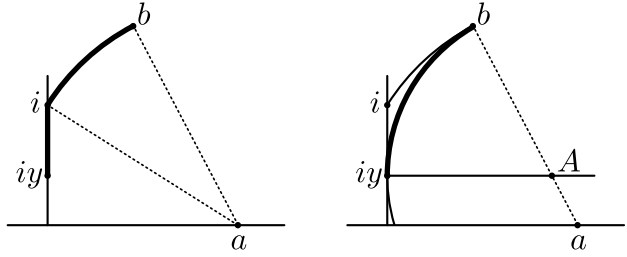}
  \caption{We replace the subarc from $iy$ to $b$ by a subarc of a circle centered at $A$.}
  \label{fig:DeltaC}
 \end{figure}

 We will replace the subarc from $iy$ to $b$ by a smooth arc with the same initial and final tangent vectors. As long as the distance from $iy$ to $i$ is at most 1, the fact that $l(\alpha), l(\beta) \geq 2$ means we can do this to both corners of $\gamma'$ at the same time to get the nearby $C^2$ curve $\gamma_c$.
 
 For each $\theta$, there is a unique $y = y(\theta)$ so that there is a Euclidean circle with center inside $\h$ that is tangent to $\tilde \gamma'$ at both $iy$ and at $b$. Moreover, we can compute that as $\theta$ goes to zero, $y(\theta)$ approaches $\frac{1}{2e-1}$. So for all $\epsilon$ small enough, the hyperbolic distance between $iy$ and $i$ is smaller than 1.
 
%
%
 
 Suppose the circle has center at $A$ and Euclidean radius $\rho$. If it meets the real axis at angle $\phi$, then its hyperbolic curvature is $|\cos(\phi)|$ \cite[Lemma 3]{GR85}. So we see that its curvature is $\frac y \rho$ at each point. By explicitly computing $y$ and $\rho$, one can show that the curvature goes to zero as $\theta$ (and $\epsilon$) go to zero.

%
%
%

 We replace each corner of $\tilde \gamma'$ in this way. This gives us our piecewise $C_2$ curve $\tilde \gamma_c$ whose curvature at each point goes to zero uniformly as $\epsilon$ goes to zero. Note that $\tilde \gamma_c$ projects down to a piecewise $C^2$ closed curve $\gamma_c$ in $S$. By \cite{Leininger06}, the distance from $\tilde \gamma_c$ to $\tilde \gamma$ is at most $f(\epsilon)$, where $f(\epsilon)$ is a continuous function with $f(0) = 0$. 
 
 By construction, the distance from $\tilde \gamma_c$ to $\tilde \gamma'$ is at most $\epsilon \sinh(1)$, as the circle segment is contained in the geodesic triangle with vertices at $iy, i$ and $b$. Thus, for all $\epsilon$ small enough, the distance from $\tilde \gamma'$ to $\tilde \gamma$ is at most
 \[
  d(\tilde \gamma, \tilde \gamma') \leq \epsilon \sinh(\rho) + f(\epsilon)
 \]
 In particular, this distance approaches 0 as $\epsilon$ goes to 0. Note that if $l(\beta) < 2$, then the same argument implies that $d(\tilde \gamma, \tilde \gamma') \leq \epsilon (\sinh(\rho) + \sinh(2)) + f(\epsilon)$. This quantity also goes to zero with $\epsilon$. So in either case, there is a function $D(\epsilon)$ so that $\gamma$ must $D(\epsilon)$ fellow travel $\gamma'$, with $\lim_{\epsilon \rightarrow 0} D(\epsilon) = 0$.

 We now show the following: For any $\epsilon > 0$ there is an $R>0$ so that for any geodesic arc $\alpha$ there exists a geodesic arc $\beta$ so that $l(\beta) \leq R$ and we can form a closed curve $\gamma' = \alpha \circ \beta$ with angle deficit at most $\epsilon$ at each corner.
 
 Take a unit speed parameterization $\alpha : [0,L] \rightarrow \S$.  Let $v = \alpha'(L)$ be its tangent vector in the unit tangent bundle $T_1(\S)$ of $\S$. Let $f_t$ denote geodesic flow on $T_1(\S)$ for time $t$, and let $r_\theta$ denote rotation by angle $\theta \in [-\pi, \pi]$. 
 
 We get coordinates in a small neighborhood about $v$ by assigning each vector $z$ the triple $(\theta, t, \phi)$, where $z = r_\phi \cdot f_t \cdot r_\theta (v)$. Let $N(v,\epsilon)$ be the set of all vectors with coordinates $(\theta, t, \phi)$ so that $|\theta|, |\phi| < \frac \epsilon 2$ and $0 < t < \frac 12 inj(X)$, where $inj(X)$ denotes the injectivity radius of $X$. (Figure \ref{fig:EpsilonTNbhd}). 
 \begin{figure}[h!]
  \centering 
  \includegraphics{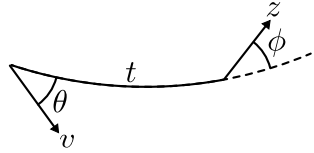}
  \caption{The set $N(v,\epsilon)$ contains all vectors $z$ with $|\theta|, |\phi| < \frac \epsilon 2$ and $t < \frac 12 inj(X)$.}
  \label{fig:EpsilonTNbhd}
  \end{figure}

 Set $w = \alpha'(0) \in T_1(\S)$. Fix any $\epsilon > 0$.  Since $N(v, \epsilon)$ is a set of positive Lebesgue measure, the mixing of the geodesic flow implies that there is some $z \in N(v, \epsilon)$ and $T > 0$ so that
 \[
 w \in N(f_T(z),\epsilon)
 \]
 (as in Figure \ref{fig:AlphaBeta}).
 
 Then if $z$ has coordinates $(\theta, t, \phi)$ near $v$ and $w$ has coordinates $(\theta', t', \phi')$ near $f_T(z)$, we can let $\eta$ be the piecewise geodesic arc whose lift to $T_1(\S)$ is given by the parameterization
 \[
  \eta'(t) = \left \{ 
  \begin{array}{ll}
   f_s \cdot r_\theta (v) & \text{ if } 0 \leq s < t\\
   f_{s-t} (z) & \text{ if } t \leq s < T + t\\
   f_{s-t -T} \cdot r_{\theta'} (f_T(z)) & \text{ if } T + t\leq s \leq T + t + t'
  \end{array}
  \right .
 \]
 (Figure \ref{fig:AlphaBeta}.) That is, we start at $v$ and flow in the direction $r_\theta(v)$ for time $t$. Then we flow in the direction of $z$ for time $T$, and lastly, in the direction of $r_{\theta'}(f_T(z))$ for time $t'$. Note that $r_\theta(v)$ and $z$ lie over the same point in $\S$, as do $f_T(z)$ and $r_{\theta'}(f_T(z))$. Thus, this defines a closed curve in $\S$.
 \begin{figure}[h!]
  \centering 
  \includegraphics{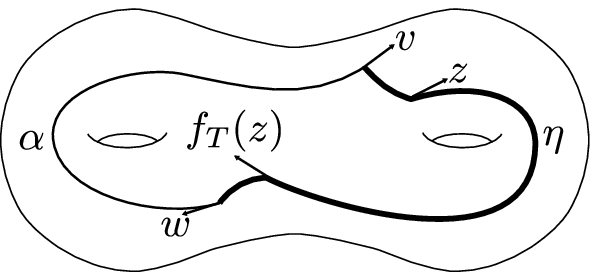}
  \caption{}
  \label{fig:AlphaBeta}
 \end{figure}

Let $\beta$ be the arc freely homotopic to $\eta$ relative its endpoints. Then $|\theta|, |\theta'|, |\phi|, |\phi'| < \frac \epsilon 2$ implies that the angle deficit at each of the two points where $\alpha$ meets $\beta$ is at most $\epsilon$. By the triangle inequality, $l(\beta) \leq T + 2$.

 In fact, there is a continuous function $T(\cdot, \cdot, \cdot) : T_1(\S) \times T_1(\S) \times \R^+ \rightarrow \R^+$ so that $T \leq T(v,w,\epsilon)$. 
Moreover, because $T_1(\S)$ is compact, there is a continuous function $T : \R^+ \rightarrow \R^+$ so that
\[
 T(v,w,\epsilon) \leq T(\epsilon)
\]
This is the only place where we use that $\S$ is compact. In particular, for any $\alpha$ and $\epsilon$, there is an arc $\beta$ of length at most $R(\epsilon) = T(\epsilon) + 2$ so that we can concatenate $\alpha$ and $\beta$ into a closed curve with angle deficit at most $\epsilon$ at its corners.

Choose $\epsilon$ small enough so that $D(\epsilon) \leq 1$, where $D(x)$ is the function we defined in the first part of this proof. Let $R = R(\epsilon)$. Then for every arc $\alpha$, there is a closed geodesic $\gamma$ so that
\[
 l(\gamma) \leq l(\alpha) + R
\]
and $\gamma$ 1-fellow travels $\alpha$.
\end{proof}

The previous claim says we can approximate any arc $\alpha$ by a closed geodesic $\gamma$ of roughly the same length. The next claim allows us to estimate the self-intersection number of $\gamma$.

\begin{claim}
\label{cla:ArcIntersection}
 If $\alpha$ and $\beta$ are two geodesic arcs with $l(\alpha) \leq L_\alpha$ and $l(\beta) \leq L_\beta$, then
 \[
  \# \alpha \cap \beta \leq \kappa L_\alpha L_\beta
 \]
 where we require  $L_\alpha, L_\beta \geq 1$, and $\kappa$ is a constant that depends only on $X$.
\end{claim}
\begin{proof}
 The proof is almost exactly the same as the proof of \cite[Theorem 1.1]{Basmajian13}, where Basmajian proves this result in the case where $\alpha$ and $\beta$ are the same (non-simple) closed geodesic. We recreate it here for completeness. 
 
 Take a pants decomposition $\Pi$ of $\S$. Further cut each pair of pants into two congruent right-angled hexagons. So each hexagon has boundary edges that lie on curves in $\Pi$, and seam edges that join curves in $\Pi$. 
 
 The hexagon decomposition cuts the arc $\alpha$ into segments, which are maximal subarcs that lie in a single hexagon, and the same is true for $\beta$. Note that if a hexagon $h$ has $n$ $\alpha$-segments and $m$ $\beta$-segments, then the total number of intersections between $\alpha$ and $\beta$ in $h$ is at most $nm$. This is because each hexagon is simply connected and convex, so any pair of segments intersects at most once. Therefore, if $\alpha$ has $N_\alpha$ total segments, and if $\beta$ has $N_\beta$ total segments, then
 \[
  \# \alpha \cap \beta \leq N_\alpha N_\beta
 \]

 So we just need to bound $N_\alpha$ and $N_\beta$ in terms of $l(\alpha)$ and $l(\beta)$, respectively. We will say a \textbf{full segment} is any segment of $\alpha$ or $\beta$ that does not contain an endpoint of that arc. Take three consecutive full segments $\sigma_1, \sigma_2, \sigma_3$ of $\alpha$. Then by the argument in \cite[Step 2, Section 3]{Basmajian13}, we have that
 \[
  l(\sigma_1) + l(\sigma_2) + l(\sigma_3) \geq C
 \]
 where $C$ depends only on the metric $X$. The total length of all full segments of $\alpha$ is at most $l(\alpha)$. Thus, its number of full segments is at most $\frac{3l(\alpha)}{C}$. The arc $\alpha$ only has two segments that are not full: they are the ones that contain its endpoints. So we have
 \[
  N_\alpha \leq 2 + \frac{3l(\alpha)}{C}
 \]
 and likewise,
 \[
  N_\beta \leq 2 + \frac{3l(\beta)}{C}
 \]
 Therefore,
 \[
  \# \alpha \cap \beta \leq \left (2 + \frac{3l(\alpha)}{C} \right ) \left (2 + \frac{3l(\beta)}{C} \right )
 \]

 If $l(\alpha) \leq L_\alpha$ and $l(\beta) \leq L_\beta$, and if we assume $L_\alpha, L_\beta \geq 1$, we have that 
 \[
  \# \alpha \cap \beta \leq (2 + \frac 3C)^2 L_\alpha L_\beta
 \]
 Thus, the claim holds for $\kappa = (2 + \frac 3C)^2$, which is a constant depending only on $X$.
\end{proof}

\begin{proof}[Proof of Lemma \ref{lem:ApproxArcs}]
Assume that $l(\alpha) > 3$. Then Claim \ref{cla:ClosingAlpha} says there is a closed geodesic $\gamma$ that 1-fellow travels $\alpha$. By construction, $\gamma$ is freely homotopic to a concatenation $\alpha \circ \beta$, where $\beta$ is a geodesic arc of length at most $R$, for $R$ depending only on the metric $X$. So Claim \ref{cla:ArcIntersection} allows us to estimate the self-intersection number of $\gamma$. 

Assuming $R \geq 1$, Claim \ref{cla:ArcIntersection} implies that
\[
 \# \alpha \cap \beta \leq \kappa R l(\alpha)
\]
and 
\[
 \# \beta \cap \beta \leq \kappa R^2
\]
Thus, if we assume $l(\alpha) < L$, 
\begin{align*}
 i(\gamma, \gamma) & \leq i(\alpha, \alpha) + i(\alpha, \beta) + i(\beta,\beta)\\
                   & \leq i(\alpha, \alpha) + \kappa R l(\alpha) + \kappa R^2\\
                   & \leq K + \kappa R (L + R)
\end{align*}
So if $L \geq R$, then $i(\gamma, \gamma) \leq K + 2 \kappa R L$. Setting $d = 2 \kappa R$, we get Lemma \ref{lem:ApproxArcs}.
\end{proof}

\section{Covering by neighborhoods of closed curves}
\label{sec:Covers}
The proofs of Theorems \ref{thm:HausDim} and \ref{thm:NWD} come from covering our sets of complete geodesics with neighborhoods of closed geodesics. For any function $f(l)$, we give an infinite collection of finite open covers $\{\C_n(f) = C_n\}$. Then for any $L$, $\C_n$ is a cover of $\Im \G(f,L)$ for all $n$ large enough. Moreover, the union of any infinite subsequence over these covers gives a cover of $\Im \G(f)$. 


Specifically, for each $\gamma \in \G^c$, let $N_\epsilon(\gamma)$ be an $\epsilon$-neighborhood of $\Im \gamma$. Then if $\H \subset \G^c$ is any collection of closed geodesics, we let
   \[
    N_\epsilon (\H ) = \bigcup_{\gamma \in \H} N_\epsilon(\gamma)
   \]
We define the cover $\C_n$ by 
\[
 \C_n = N_{\epsilon(n)} \left ( \G^c_n(c_X \cdot f( n))  \right )
\]
where $\epsilon(n) = 2e^{-n/4}$, and $c_X = 2 + d/2$, for the constant $d$ defined in Lemma \ref{lem:ApproxArcs}. That is, $\C_n$ is a finite collection of $\epsilon(n)$-neighborhoods of closed geodesics of length at most $n$, with at most $c_X f(n)$ self-intersections.

\subsection{Finite covers}
Recall that
\[
 \G(f,L) = \{ \gamma \in \G \ | \ \# \gamma|_{[a,a+l]} \cap \gamma|_{[a,a+l]} \leq f(l), \forall l \geq L\}
\]
In other words, this is the set of complete geodesics $\gamma$ so that all length $l$ subarcs have self-intersection number at most $f(l)$, for all $l \geq L$. Because we impose this regularity on the self-intersection function of geodesics in $\Im \G(f,L)$, we can show that each $\C_n$ is a cover of $\G(f,L)$, as long as $n$ is large enough.

Observe that if $f(x) \leq g(x)$, then $\G(f) \subset \G(g)$, and in fact, $\G(f,L) \subset \G(g,L)$ for all $L$. So in this section, we assume without loss of generality that $f(l) \geq l$ for all $l$.




\begin{lem}
 \label{lem:FiniteCover}
 Suppose $f(l) \geq l$ for all $l$. Then for each $L$, there is an $N > 0$ so that for all $n \geq N$, $\C_n$ is a cover of $\Im \G(f,L)$.
 \end{lem}

 \begin{proof}
 Let $x \in \Im \G(f,L)$. Then there is a $\gamma \in \G(f,L)$ parameterized by $\gamma : \R \rightarrow \S$ so that $x = \gamma(t_x)$ for some time $t_x \in \R$. Choose $l \geq L$. Let $\alpha_{x,l} = \gamma|_{[t_x - \frac l 2, t_x + \frac l2]}$ be the length $l$ subarc of $\gamma$ centered at $x$. Then because $\gamma \in \G(f,L)$,
 \[
  \# \alpha_{x,l} \cap \alpha_{x,l} \leq f(l)
 \]
 
By Lemma \ref{lem:ApproxArcs}, as long as $l \geq d$, there is a closed geodesic $\delta \in \G^c_{2l}(f(l)+dl)$ that 1-fellow travels $\alpha_{x,l}$, where $d$ is a constant depending only on $x$. Since $f(l) \geq l$, we have, in fact, that $\delta \in \G^c_{2l}(c_X f(l))$, where $c_X = 1+ d$.

Lift $\delta$ and $\alpha_{x,l}$ to curves $\tilde \delta$ and $\tilde \alpha_{x,l}$ in the universal cover such that the endpoints of $\tilde \alpha_{x,l}$ are at most distance 1 away from $\tilde \delta$. The midpoint of $\tilde \alpha_{x,l}$ is a lift $\tilde x$ of $x$. Thus,
\[
 d(\tilde x, \tilde \delta) < 2 e^{-l/2}
\]
To see this, drop perpendiculars from $\tilde x$ and from an endpoint of $\tilde \alpha_{x,l}$ down to $\tilde \delta$. This forms a quadrilateral with 3 right angles (called a Lambert quadrilateral.) By properties of Lambert quadrilaterals, $d(\tilde x, \tilde \delta) \leq \sinh(1)/ \cosh(\frac l 2)$. Since $\sinh(1) < 2$ and $\cosh(\frac l 2) > e^{l/2}$, we have our inequality (Figure \ref{fig:Lambert}).
\begin{figure}[h!]
 \centering 
 \includegraphics{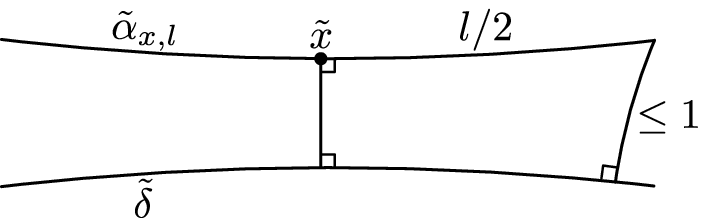}
 \caption{}
 \label{fig:Lambert}
\end{figure}

Let $N = \max\{2L, d\}$. Then for each $n \geq N$, and for each $x \in \Im \G(f,L)$, there is a $\delta \in \G^c_n(c_X f(n))$ so that $x \in N_{\epsilon(n)}(\delta)$ for $\epsilon(n) = 2 e^{-n/4}$. In other words, for all $n \geq N$, $\C_n$ is a cover of $\Im \G(f, L)$.
\end{proof}

\subsection{Infinite covers}
To cover $\Im G(f)$, we need to take the union of infinitely many covers in the sequence $\{\C_n\}$. Once again, observe that if $f(l) \leq g(l)$ for two functions $f$ and $g$, then $\G(f) \subset \G(g)$. So we can assume without loss of generality that $f(l) \geq l$ for all $l$, and that $f(l)$ is increasing in $l$.

\begin{lem}
\label{lem:NestedCover}
Suppose $f(l)$ is an increasing function so that $f(l) \geq l$. Then for each $N >0$, $\bigcup_{n = N}^\infty \C_n $ is a cover of $\Im \G(f)$.
\end{lem}
\begin{proof}
 Let $x \in \Im \G(f)$. Then there is a $\gamma \in \G(f)$ parameterized by $\gamma : \R \rightarrow \S$ so that $x = \gamma(t_x)$ for some time $t_x \in \R$.
 
 Recall that $\gamma_l = \gamma|_{[-\frac l 2, \frac l2]}$ is the length $l$ subarc of $\gamma$ centered at $\gamma(0)$. Because $\gamma \in \G(f)$, there is some length $l_0$ depending on $\gamma$ so that for all $l \geq l_0$,
 \[
  \# \gamma_l \cap \gamma_l \leq 2 f(l)
 \]
 For each $l$, let $\alpha_{x,l} = \gamma|_{[t_x - \frac l 2, t_x + \frac l2]}$ be the length $l$ subarc of $\gamma$ centered at $x$. Then $\alpha_{x,l} \subset \gamma_{l+t_x}$ for each $l$. 
  
  Let $l \geq \max\{l_0, t_x, d\}$, where $d$ is the constant from Lemma \ref{lem:ApproxArcs}. Then $l+t_x \leq 2l$. Since we assume that $f(l)$ is increasing, this means that $f(l+t_x) \leq f(2l)$. Thus,
 \[
  \# \alpha_{x,l} \cap \alpha_{x,l} \leq 2f(2l)
 \]
 
 By the same argument as in the proof of Lemma \ref{lem:FiniteCover}, there is a closed geodesic $\delta \in \G^c_{2l}(2f(2l) + dl)$ so that 
 \[
  d(x, \delta) \leq 2e^{-l/2}
 \]
Since we assume $f(l) \geq l$, we have, in fact, that $\delta \in \G^c_{2l}(c_X f(2l))$, where $c_X = 2 + d/2$.

Then for each $x$, and for each $n \geq \frac 12 \max\{ l_0, t_x, d\}$, there is a $\delta \in \G^c_n(c_X f(n))$ so that $x \in N_{\epsilon(n)}(\delta)$ for $\epsilon(n) = 2 e^{-n/4}$. In other words, $ \bigcup_{n \geq N} \C_n$ is a cover of $\Im \G(f)$ for each $N$.
\end{proof}

%
%
%
%
%

\section{Counting the approximating closed curves}
\label{sec:CountingCoverSize}


We will define the Lebesgue (or Hausdorff) measure of any cover $\C_n$ to be the measure of the union of the open sets in $\C_n$. We will eventually wish to show that the measures of these covers go to zero as $n$ goes to infinity. Recall that we defined each cover $\C_n$ as the collection of open neighborhoods about geodesics in $\G^c_n(c_X f(n))$. So to show that these measure go to zero, we need to bound the number of closed geodesics in these sets.

\begin{lem}
\label{lem:ClosedGeodesicGrowth}
 If $f(n) \leq (k n)^2$, then
 \[
  \#\G^c_n(c_Xf(n)) = o(\frac 1 n e^{\mu(k) n})
 \]
 where $\lim_{k \rightarrow 0} \mu(k) = 0$, and $c_X$ is the constant depending only on $X$ defined in Section \ref{sec:Covers}.
\end{lem}
In fact, we will show that $\mu(k) = a_X k \log(a_X(1 + \frac 1 k))$, where $a_X$ depends only on $X$.
First we will bound $\# \G^c_L(K)$ for any $L$ and $K$, and then we will set $L = n$ and consider the case where $K(n) \leq (k n)^2$.

\begin{claim}
Let $\S$ be a closed, genus $g$ surface. For any $L$ and $K$, we have
 \[
  \#\G^c_L(K) \leq p(L) \left( a_X \frac{L}{\sqrt K} + a_X \right)^{a_X \sqrt K} 
 \]
where $p(L)$ is a polynomial in $L$, and $a_X$, as well as the coefficients of $p(L)$, depend only on the metric $X$.
\end{claim}
In fact, one can look carefully at the argument in \cite[Lemma 5.6]{Mir16} to see that $p(L)$ can be replaced by a polynomial in $K$ times $L^{6g-6}$. However, the formula in Lemma \ref{lem:ClosedGeodesicGrowth} is easier to prove, and suffices for Lemma \ref{lem:ClosedGeodesicGrowth}.
\begin{proof}
 Let $\Mod_\S$ denote the mapping class group of $\S$. Then $\Mod_\S$ acts on $\G^c$, preserving self-intersection number. For each $\gamma \in \G^c$ let $\Mod_\S \cdot \gamma$ denote its orbit. Let $\O(\cdot, K)$ be the set of orbits of curves with at most $K$ self-intersections:
 \[
  \O(\cdot, K) = \{\Mod_\S \cdot \gamma \ | \ i(\gamma, \gamma) \leq K\}
 \]
If $\gamma$ has $K$ self-intersections, then the shortest curve in $\Mod_\S \cdot \gamma$ has length between $c_1 \sqrt K$ and $ c_2 K$ for some constants $c_1$ and $c_2$. In fact, there exist such constants, depending only on $X$, for which these bounds are tight \cite{AGPS16,Basmajian13,Gaster15}. Thus, if $L$ is small enough, then not all $\Mod_\S$ orbits contain curves of length at most $L$. So we let 
\[
 \O(L,K) = \{\Mod_\S \cdot \gamma \ | \ \Mod_\S \cdot \gamma \cap \G^c_L(K) \neq \emptyset \}
\]
be those orbits that contain curves of length at most $L$.

In \cite{SapirOrbits}, we show that 
\[
 \# \O(L,K) \leq \left ( a_X \frac{L}{\sqrt K} + a_X \right)^{a_X \sqrt K}  
\]
for a constant $a_X$ depending only on $X$.

Since we have a bound on the number of orbits, we just need a bound on the number of curves in each orbit. For each $\gamma$, let 
\[
 s(L, \gamma) = \{ \gamma' \in \Mod_\S \cdot \gamma \ | \ l(\gamma') \leq L\}
\]
In \cite{Mir16}, Mirzakhani shows that $s(L,\gamma)$ grows asymptotically like $a_{\gamma, X} L^{6g-6}$, where the constant $a_{\gamma, X}$ depends on the $\Mod_g$ orbit of $\gamma$, and on $X$. The dependence of $a_{\gamma, X}$ on $\gamma$ is difficult to determine. As mentioned above, a careful analysis of \cite[Lemma 5.6]{Mir16} should imply that $s(L,\gamma) \leq p(K) L^{6g-6}$, where $p(K)$ is a polynomial in $K$ whose coefficients depend only on the metric $X$. However, there is a faster way to see that 
\[
\# s(L,\gamma) \preccurlyeq L^{30g-12}
\]
where the constant depends only on $X$.

First, suppose $\gamma$ is a filling curve on a closed, genus $g$ surface. Then we the proof of \cite[Lemma 2.2]{SapirU} implies that
\[
 \# s(L,\gamma) \preccurlyeq L^{6g-6}
\]
where the constant depends only on $X$.

Now suppose $\gamma$ only fills a proper subsurface $T \subset \S$. Then by \cite[Proposition 2.7]{SapirU}, $l(\partial T) \leq 2l(\gamma)$. So for any $g \in \Mod_\S$ with $l(g \cdot \gamma) \leq L$ we have $l(g \cdot \partial T) \leq 2L$. The number of simple closed curves of length at most $2L$ on $\S$ is at most $b_X (2L)^{6g-6}$, where $b_X$ is a constant depending only on $X$.

So we can now fix a subsurface $T$ of $\S$, and count all curves in $\Mod_\S \cdot \gamma$ that fill $T$:
\[
 s(L, \gamma, T) = \{\gamma' \in \Mod_\S \cdot \gamma \ | \ l(\gamma') \leq L, \gamma' \text{ fills } T\}
\]

Double $T$ across its boundary. This gives a new surface $Q = T \cup \bar T$, where $\bar T$ is the complement of $T$ in $Q$. The metric on $Q$ is obtained by doubling the metric on $T$. Moreover, if $\gamma' \in \Mod_\S \cdot \gamma$ lies in $T$, then it has a mirror image $\bar \gamma'$ that lies in $\bar T$. Finally, for each component $\alpha$ of $\partial T$, there is a curve $\beta$ that intersects $\alpha$ minimally with $l(\beta) \preccurlyeq l(\gamma')$, where the constant depends only $X$ (\cite[Proposition 2.7]{SapirU}). Let $\eta$ be the union of all such curves $\beta$ together with $\partial T$. Then consider the curve
\[
 \delta = \gamma' \cup \bar \gamma' \cup \eta
\]
By construction, $\delta$ fills the closed surface $Q$. If $\S$ had genus $g$, then the genus of $Q$ is at most $4g$. (Really, the genus of $Q$ is at most $4g-1$, but this slight improvement in the upper bound leads to messier formulae, which are still not tight.)
%
%

So we can count curves in $ s(L, \gamma, T)$ on $\S$ by counting curves in $s(L, \delta)$ on $Q$, instead. In fact, since $\delta$ fills a closed surface, we get
\[
 s(L,\gamma, T) \leq s(L,\delta) \preccurlyeq L^{24g - 6}
\]
Thus,
\[
 s(L,\gamma) \preccurlyeq L^{6g-6} \cdot L^{24g - 6} = L^{30g - 12}
\]
for constants that depends only on the metric $X$.

Combined with the orbit counting result, we get that
\[
 \# \G^c_L(K) \leq p(L) \left ( a_X \frac{L}{\sqrt K} + a_X \right)^{a_X \sqrt K}
\]
where $p(L)$ is a polynomial in $L$ of degree $30g-12$, whose constants depend only on $X$, and $a_X$ depends only on $X$.
\end{proof}

\begin{proof}[Proof of Lemma \ref{lem:ClosedGeodesicGrowth}]
Now set $L = n$ and suppose $K \leq (k n)^2$ for some $k$. Then we are ready to show that
\[
  \# \G^c_n(c_X f(n)) = o(\frac 1 n e^{\mu(k) \cdot n})
\]
where
\[
 \mu(k) = \frac 12  a_X k\ln(\frac{a_X}{k} + a_X)
\]

Because $p(n) = o(\frac 1 n e^{\frac 12 \mu(k) n})$ for any polynomial $p(n)$ and any positive $\mu(k)$, we only need to find a function $\mu(k)$ so that 
\[
 \left ( a_X \frac{n}{\sqrt K} + a_X \right)^{a_X \sqrt K} = o(e^{\frac 12 \mu(k) n})
\]
whenever $K \leq (k n)^2$.

For any fixed $n \geq 1$, we have that $\left ( a_X \frac{n}{\sqrt K} + a_X \right)^{a_X \sqrt K}$ is an increasing function in $K$, since we may assume that $a_X \geq e$.
Since we assume that $K \leq (k n)^2$, this means
\[
 \left ( a_X \frac{n}{\sqrt K} + a_X \right)^{a_X \sqrt K} \leq \left ( a_X \frac 1 k+ a_X \right)^{a_X \cdot kn} = e^{n \cdot a_X k\ln(\frac{a_X}{k} + a_X)}
\]
In other words, if $K = K(n)\leq (k n)^2$, then
\[
 \left ( a_X \frac{n}{\sqrt K} + a_X \right)^{a_X \sqrt K} = o(e^{1/2 \mu(k) \cdot n})
\]
for $\mu(k) = 4 a_X k\ln(\frac{a_X}{k} + a_X)$. Lastly, note that $\lim_{k \rightarrow 0} \mu(k) =0$.

\end{proof}

\section{Nowhere density}
\label{sec:NWD}
A set $U \subset \S$ is nowhere dense if its closure has an empty interior. In particular, $U$ is nowhere dense if, for any open ball $\B$, $\B \setminus U$ contains a non-empty open set. 

We will show this is the case for $\Im \G(f,L)$. In particular, Lemma \ref{lem:FiniteCover} gives us a family $\{\C_n\}$ of covers of $\Im\G(f,L)$, where each $\C_n$ is a finite collection of regular neighborhoods about closed geodesics. Below we show that these covers have arbitrarily small Lebesgue measure. (The Lebesgue measure of $\C_n$ is defined to be the measure of the union of all elements of $\C_n$.)  This will imply that $\Im \G(f,L)$ is nowhere dense.

\begin{proof}[Proof of Theorem \ref{thm:NWD}]
 By Lemma \ref{lem:FiniteCover}, there is an $N$ depending only on $L$ so that for all $n \geq N$, $\C_n$ is a cover of $\Im \G(f,L)$. We wish to estimate the Lebesgue measure of $\C_n$. Recall that $\C_n$ is the set of $\epsilon(n)$-neighborhoods of the closed geodesics in $\G^c_n(c_Xf(n))$, where $\epsilon(n) = 2e^{-n/4}$. 
 
 Let $\lambda(A)$ denote the Lebesgue measure of any subset $A \subset \S$. If $\gamma \in \G^c_n(c_X f(n))$, then $l(\gamma) \leq n$. So for all $\epsilon(n)$ small enough, the measure of $N_{\epsilon(n)}(\gamma)$ is bounded above by
 \[
  \lambda \left ( N_{\epsilon(n)}(\gamma) \right ) \leq 5  n e^{-n/4}
 \]
 By Lemma \ref{lem:ClosedGeodesicGrowth}, if $f(n) \leq (k n)^2$, then $\# \G^c_n(c_X f(n)) = o \left ( \frac 1n e^{\mu(k) n} \right )$. So
 \[
  \lambda (\C_n) = o \left (e^{(\mu(k) - \frac 14) n} \right )
 \]

 We have that $\lim_{k \rightarrow 0} \mu(k) = 0$. So there is some $k_0$ so that for all $k < k_0$, $\mu(k) < \frac 14$. Then for all $k \leq k_0$, 
 \[
 \lim_{n \rightarrow \infty} \lambda(\C_n) = 0
 \]
 
 Suppose $k \leq k_0$. Choose any open ball $\B \subset \S$. Choose $n$ so that $\lambda(\C_n) < \frac 12 \lambda(\B)$. Then $\B$ is crossed by finitely many elements of $\C_n$. The elements of $\C_n$ are regular neighborhoods of closed geodesics, so our choice of $n$ guarantees that $\B \setminus \C_n$ has non-empty interior. But $\C_n$ is an open cover of $\Im \G(f,L)$. So $\B \setminus \Im \G(f,L)$ has non-empty interior, as well. Therefore, $\Im \G(f,L)$ is nowhere dense for all $L$ and all functions $f$ with $f(l) \leq (kl)^2$.
 
\end{proof}

\section{Hausdorff dimension}
\label{sec:HausDimProof}

\begin{proof}[Proof of Theorem \ref{thm:HausDim}]

The Hausdorff dimension of a set is defined as follows. Given a subset $G$ of a metric space $X$, let $\C = \{\B(x_i, r_i)\}$ be a countable cover of $G$ by metric balls centered at $x_i$ and of radius $r_i$, for each $i$. We define the $h$-dimensional Hausdorff measure of $\C$ to be $\nu_h(\C) = \sum r_i^h$. The $h$-dimensional Hausdorff measure of a set $G$ is defined as
\[
 \nu_h(G) = \inf_\C \nu_h(\C)
\]
where the infimum is taken over all such covers of $G$. Then the Hausdorff dimension of $G$ is defined to be 
\[
 \dim_H(G) = \inf\{h \ | \ \nu_h(G) = 0\}
\]

By Lemma \ref{lem:NestedCover}, infinite unions of the covers $\C_1, \dots, \C_n, \dots $ cover $\Im \G(f)$. These are covers by regular neighborhoods of closed geodesics, but we can use them to build covers of $\Im \G(f)$ by metric balls. In fact, for each $n$, we can build new cover $\C^H_n$, which is a collection of balls whose union contains the union of open sets in $\C_n$. Note that for all $\epsilon(n)$ small enough, we can cover the $\epsilon(n)$-regular neighborhood of any $\gamma \in \G^c_n(c_X f(n))$ by $2 \frac{n}{\epsilon(n)}$ balls of radius $2 \epsilon(n)$. So let $\C_n^H$ be the union of all these balls for each open set in $\C_n$. We will use the collection $\{\C^H_n\}$ of these covers to bound the Hausdorff dimension of $\Im \G(f)$.


Then Lemmas \ref{lem:NestedCover} and \ref{lem:ClosedGeodesicGrowth} allow us to estimate the Hausdorff $h$-volume of $\Im \G(f)$ by estimating the volume of $\C^H_n$. (The volume of a cover is defined to be the volume of the union of all elements of the cover.)

Let $\C^H_n$ be the collection of metric balls defined above. First, we find a condition on $h$ so that 
\[
 \lim_{n \rightarrow \infty} \nu_h(\C^H_n) = 0
\]
Each ball in $\C^H_n$ has radius $2 \epsilon(n) = 4 e^{-\frac n 4}$. 
Each closed geodesic $\gamma \in \G^c_n(c_X f(n))$ has length $n$, so it contributes $\frac{2n}{\epsilon(n)} = n e^{n/4}$ balls to the cover. 
So the total Hausdorff $h$-volume of $\C^H_n$ is bounded above by
\begin{align*}
 \nu_h (\C^H_n) & = \sum_{\G^c_n(c_X f(n))} n e^{\frac n4} \left (4 e^{-\frac n 4} \right)^h \\
   & = \sum_{\G^c_n(c_X f(n))} 4^h n e^{\frac n4 (h-1)}
\end{align*}
If $f(n) \leq (k n)^2$, then by Lemma \ref{lem:ClosedGeodesicGrowth}, the number of these closed geodesics in $\G^c_n(c_X f(n))$ grows like $o \left(\frac 1n e^{\mu(k) n} \right )$, where $\lim_{n\rightarrow 0} \mu(k) = 0$. Since $4^h$ is a constant for each $h$,
\[
 \nu_h (\C^H_n)  = o \left(e^{n (\mu(k) - \frac{h-1}{4})} \right )
\]
In particular, $\lim_{n \rightarrow \infty} \nu_h(\C^H_n) = 0$ whenever $h > 4 \mu(k) + 1$.

Suppose $h > 4 \mu(k) + 1$. Then there is a subsequence $\{n_i\}$ so that $\nu_h(\C^H_{n_i}) \leq 2^{-i}$ for each $i$. By Lemma \ref{lem:NestedCover}, any infinite subsequence of $\{\C^H_n\}$ covers $\Im \G(f)$. In particular, for any $N \geq 0$, $\cup_{i=N}^\infty \C^H_{n_i}$ covers $\Im \G(f)$. Thus,  whenever $h > 4 \mu(k) + 1$,
\[
 \nu_h(\Im \G(f)) \leq 2^{-N}
\]
for any $N$. In other words, the Hausdorff dimension of $\Im \G(f)$ is at most $4 \mu(k) + 1$. 

Furthermore, suppose $f(l) = o(l^2)$. Then $\G(f) \subset \cap_{k=1}^{\infty} \G(f_k)$, for $f_k = kl^2$. So in this case, the Hausdorff dimension of $\Im \G(f)$ is 1.

\end{proof}

 \bibliographystyle{alpha}
  \bibliography{recount}

\end{document}